%% file: ms.tex
\documentclass[a4paper,10pt]{article}
\usepackage{mypackages}
\usepackage{mymacros}
\usepackage{mystyle}

\title{A pullback diagram in the coarse category}
\author{Elisa Hartmann\thanks{Department of Mathematics Institute for Algebra and Geometry, Karlsruhe Institute of Technology, Germany}}

\begin{document}

\maketitle

\begin{abstract}
This paper studies the asymptotic product of two metric spaces. It is well defined if one of the spaces is visual or if both spaces are geodesic. In this case the asymptotic product is the pullback of a limit diagram in the coarse category. Using this product construction we can define a homotopy theory on coarse metric spaces in a natural way. We prove that all finite colimits exist in the coarse category.

{\bf keywords:} coarse category, asymptotic product, coarse homotopy, finite colimits, metric spaces, cocomplete
\end{abstract}

\tableofcontents

\input{Introduction}
\input{Metric_Spaces}
\input{Groundwork}
\input{Topology_of_Coarse_Spaces}

\input{Homotopy2}

\input{Acyclic_Spaces}
\input{Colimits}

\bibliographystyle{halpha-abbrv}
\bibliography{mybib}

\address

\end{document}

%% file: Introduction.tex
\section{Introduction}

The area of coarse geometry studies metric spaces as geometric objects from an asymptotic viewpoint. Key results in this area include Gromov's result on groups of polynomial growth \cite{Gromov1981} and Yu's result on the Novikov conjecture \cite{Yu1998}.

The coarse category\footnote{The coarse category consists of metric spaces as objects and coarse maps modulo close as morphisms. We use the convention that metrics only take finite values.} is not equipped with a faithful functor to the category of sets and is therefore not concrete. Nonetheless finite coproducts exist in the coarse category \cite{Weighill2016},\cite{Hartmann2017a}. They can be characterized as coarse covers in which every two elements are coarsely disjoint. Most boundaries on coarse spaces preserve finite coproducts \cite{Weighill2016},\cite{Hartmann2019a},\cite{Roe2003},\cite{Hartmann2017b}.

Conversely products do not exist in the coarse category. In fact the coarse category does not even have a final object. The coarse product of two coarse spaces $X,Y$ is the product of sets $X\times Y$ equipped with the pullback coarse structure from $X,Y$ and the projection maps. This way one can define a natural coarse structure on the product, but the projections are not coarse maps. Thus although the coarse product is well defined up to coarse equivalence it is not a product in the strict category sense.

Now we introduce a pullback diagram in the coarse category, the limit of which is called the \emph{asymptotic product}. The coarse version of the point, $\R_{\ge 0}$ the non-negative real numbers, is unfortunately not a final object in the coarse category. Nonetheless we look at a pullback diagram of metric spaces
\[
\xymatrix{
& Y\ar[d]^{d(\cdot,y_0)}\\
X\ar[r]_{d(\cdot,x_0)} 
& \R_{\ge 0}.
}
\]
Here $x_0\in X,y_0\in Y$ are fixed points. The pullback of this diagram is called $X\ast Y$. It exists if the spaces $X,Y$ are nice enough as studied in Theorem~\ref{thm:pullbackexists}. Indeed we only need both $X,Y$ to be geodesic or if $X$ is not geodesic $Y$ to be visual.

If $X,Y$ are hyperbolic proper geodesic metric spaces then $X\ast Y$ is hyperbolic geodesic proper by \cite{Foertsch2003}, therefore its Gromov boundary $\partial(X\ast Y)$ is defined. The paper \cite{Foertsch2003} shows there is a homeomorphism
\[
\partial(X\ast Y)=\partial (X)\times \partial(Y).
\]
Therefore the asymptotic product is indeed a product on the Gromov boundary. Thus the terminology \emph{product} is justified.

Equipped with the asymptotic product we can define a coarse version of homotopy. Our coarse version of an interval is denoted by $I$, any visual metric space might be a candidate for such a space, we chose $I=\R_{\ge0}\times \R_{\ge0}$, the first quadrant in $\R^2$. Then a coarse homotopy is defined to be a coarse map 
\[
H:X\ast I\to Y.
\]
Here $\ast$ refers to the asymptotic product and $X,Y$ are metric spaces. There are other notions of homotopy on metric spaces which were studied in the coarse setting. The Lipschitz homotopy is due to Gromov \cite{Gromov1993}. It was shown by Roe that controlled operator $K$-theory is a Lipschitz homotopy invariant \cite{Roe1996}. Moreover Lipschitz homotopy has been used by Yu in the proof of the coarse Baum-Connes conjecture for spaces which admit a uniform embedding into Hilbert space \cite{Yu2000}. Other notions of coarse homotopy appeared in \cite{Mitchener2001,Wulff2020,Mitchener2020,Bunke2016,Sawicki2005}. In particular the homotopy theories introduced in \cite{Mitchener2020} will be shown to be equivalent to our homotopy theory if the parameter function in \cite{Mitchener2020} is chosen appropriately. This homotopy theory in turn is equivalent to the homotopy theory introduced in \cite{Bunke2016} if the two parameter functions in \cite{Bunke2016} are chosen appropriately.

The class of flasque metric spaces erases many coarse cohomology and homology theories. Theorem~\ref{thm:flasqueacyclic} shows that a flasque metric space $X$ is coarsely homotopy equivalent to the product $X\times\Z_{\ge0}$. 

The last section will be about colimits in the coarse category. The coproduct of two metric spaces is known to exist and appears in many isomorphic forms. We show the coequalizer of two coarse maps with common domain and codomain exists. This result provides a proof that finite colimits exist in the coarse category.

%% file: Metric_Spaces.tex
\section{Metric Spaces}
\label{sec:metric}

\begin{defn}
 Let $(X,d)$ be a metric space. Then the \emph{coarse structure associated to $d$} on $X$ consists of those subsets $E\s X^2$ for which
 \[
  \sup_{(x,y)\in E}d(x,y)<\infty.
 \]
 We call an element of the coarse structure an \emph{entourage}.  In what follows we assume the metric $d$ to be finite for every $(x,y)\in X^2$.
\end{defn}

\begin{defn}
 A map $f:X\to Y$ between metric spaces is called \emph{coarse} if
 \begin{itemize}
  \item $E\s X^2$ being an entourage implies that $\zzp f E$ is an entourage \emph{(coarsely uniform)};
  \item and if $A\s Y$ is bounded then $\iip f A$ is bounded \emph{(coarsely proper)}.
 \end{itemize}
 Two maps $f,g:X\to Y$ between metric spaces are called \emph{close} if
 \[
 f\times g(\Delta_X)
 \]
 is an entourage in $Y$. Here $\Delta_X$ denotes the diagonal in $X^2$.
\end{defn}

\begin{notat}
 If $S\s X\times X,A\s X$ are subsets we define 
 \[
  S[A]:=\{x\in X\mid \exists y\in A:(x,y)\in E\}.
 \]
 If $S,T\s X\times X$ then
 \[
  S\circ T:=\{(x,z)\in X\times X\mid \exists y\in X: (x,y)\in S,(y,z)\in T\}
 \]
If $R\ge 0$ we define
\[
 \Delta_R:=\{(x,y)\in X\times X\mid d(x,y)\le R\}.
\]
\end{notat}

\begin{notat}
 A map $f:X\to Y$ between metric spaces is called
 \begin{itemize}
 \item \emph{coarsely surjective} if there is an entourage $E\s Y^2$ such that 
 \[
  E[\im f]=Y;
 \]
 \item \emph{coarsely injective} if for every entourage $F\s Y^2$ the set $\izp f F$ is an entourage in $X$.
\end{itemize}
\end{notat}

\begin{rem}
 We study metric spaces up to coarse equivalence. A coarse map $f:X\to Y$ between metric spaces is a \emph{coarse equivalence} if
 \begin{itemize}
  \item There is a coarse map $g:Y\to X$ such that $f\circ g$ is close to $id_Y$ and $g\circ f$ is close to $id_X$.
 \item or equivalently if $f$ is both coarsely injective and coarsely surjective.
 \end{itemize}
 The \emph{coarse category} consists of coarse (metric) spaces as objects and coarse maps modulo close as morphisms. Coarse equivalences are the isomorphisms in this category. 
\end{rem}

The following notion is based on the notion of $\omega$-excisive \cite{Higson1993}.
\begin{defn}
 If $X$ is a metric space and $A,B\s X$ are subsets then $(A,B)$ are called a \emph{coarsely excisive pair} if $A\cup B=X$ and for each $R>0$ there is some $S>0$ such that
 \[
  \Delta_R[A]\cap\Delta_R[B]\s \Delta_S[A\cap B].
 \]
\end{defn}
If $X,Y$ are two metric spaces then unless otherwise stated we define the metric on $X\times Y$ to be
\[
 d((x_1,y_1),(x_2,y_2))=\max(d(x_1,x_2),d(y_1,y_2))
\]
for every $(x_1,y_1),(x_2,y_2)\in X\times Y$.

%% file: Groundwork.tex
\section{Coarse rays and quasigeodesic rays}

In~\cite{Kuchaiev2008} every metric space that is coarsely equivalent to $\Z_{\ge0}$ is called a coarse ray. We keep with that notation: If $X$ is a metric space a sequence $(x_i)_i\s X$ is called a \emph{coarse ray in $X$} if there is a coarsely injective coarse map $\rho:\Z_{\ge0}\to X$ such that $x_i=\rho(i)$ for every $i$.

Recall \cite[Definition~6.16]{Roe2003}:
\begin{defn}
      Let $X$ be a metric space and $C\ge 0,D>0$ be constants. A 
\emph{$(D,C)$-quasigeodesic ray} in $X$ is a map $\rho:\Z_{\ge0}\to X$ such that
\[
      \ii D|n-m|-C\le d(\rho(n),\rho(m))\le D|n-m|+C
\]
for every $n,m\in \Z_{\ge0}$.
\end{defn}

\begin{defn}
      A geodesic metric space $X$ is called \emph{large-scale visual} if there 
exist $x_0\in X$, $C\ge 0,D>0$ such that for every $y\in X$ there exists a 
$(D,C)$-quasigeodesic ray $\rho$ starting at $x_0$ with $y\in \rho(\Z_{\ge0})$.  

If $x_0\in X$ then $X$ is called \emph{visual from $x_0$} if it is geodesic and every geodesic segment starting at $x_0$ can be extended to a geodesic ray starting at $x_0$. The space $X$ is called \emph{visual} if it is visual from $x_0$ for some point $x_0\in X$.  
\end{defn}

\begin{rem}
    Recall \cite[Definition~3.A.8.a)]{Cornulier2016}: A map $\varphi:X\to Y$ 
between metric spaces is called \emph{large-scale Lipschitz} if there exist 
constants $D>0,C\ge 0$ such that
\[
      d(\varphi(x),\varphi(x'))\le D d(x,x')+C
\]
for every $x,x'\in X$. Note that a coarse map between geodesic metric spaces is already large-scale Lipschitz and every coarse equivalence between geodesic metric spaces is a quasi-isometry.
\end{rem}

\begin{lem}
      On geodesic metric spaces the property large-scale visual is coarse.
\end{lem}
\begin{proof}
      Let $\alpha: X\to Y$ be a coarse equivalence between geodesic metric 
spaces with $X$ large-scale visual. We prove $Y$ is visual. Suppose $C'\ge 0,D'>0$ are 
constants and $x_0\in X$ is a point such that every point $x\in X$ is on a 
$(D',C')$-quasigeodesic ray joining  $x_0$ to $x$. Suppose $\alpha$ is 
$K$-coarsely surjective and a $(D'',C'')$-quasi-isometric embedding.

Let $y\in Y$ be a point. Then there exists $x\in X$ with $d(\alpha(x),y)<K$. 
And there exists a $(D',C')$-quasigeodesic ray $\rho$ in $X$ starting at $x_0$ with $x=\rho(n)$ for some $n\in \N$. Then $\alpha\circ \rho$ is a 
$(D''D',D''C'+C'')$-quasigeodesic ray starting at $\alpha(x_0)$ with 
$\alpha(x)= \alpha\circ \rho(n)$. Define
\begin{align*}
\rho':\Z_{\ge0}&\to Y\\
i&\mapsto \begin{cases}
               \alpha\circ\rho(i) & i< n\\
               y & i=n\\
               \alpha\circ \rho(i) & i>n.
         \end{cases}
\end{align*}
Now
\begin{align*}
 d(\rho'(i),\rho'(j))
 &\le d(\alpha\circ\rho(i),\alpha\circ\rho(j))\\
 &\le D''d(\rho(i),\rho(j))+C''+K\\
 &\le D'D''|i-j|+C'+C''+K
\end{align*}
and in the other direction
\begin{align*}
 d(\rho'(i),\rho'(j))
 &\ge d(\alpha\circ \rho(i),\alpha\circ(j))-K\\
 &\ge D''^{-1}d(\rho(i),\rho(j))-C''-K\\
 &\ge D'^{-1}D''^{-1}d(\rho(i),\rho(j))-C'-C''-K.
\end{align*}

Thus we have shown $\rho'$ is a $(D''D',C'+C''+K)$-quasigeodesic ray starting at $\alpha(x_0)$ with $y\in \rho'(\Z_{\ge0})$.
\end{proof}

%% file: Topology_of_Coarse_Spaces.tex
\section{Asymptotic product}

\begin{lem}
If $X$ is a metric space, fix a point $x_0\in X$, then
\begin{align*}
d(x_0,\cdot):X&\to \R_{\ge 0}\\
x&\mapsto d(x_0,x)
\end{align*}
is a coarse map.
\end{lem}
\begin{proof}
We show $d(x_0,\cdot)$ is coarsely uniform: Let $R\ge 0$ be a number. If $d(x,y)\le R$ then 
  \begin{align*}
   |d(x_0,x)-d(x_0,y)|
   &=\begin{cases}
      d(x_0,x)-d(x_0,y) & d(x_0,x)\ge d(x_0,y)\\
      d(x_0,y)-d(x_0,x) & d(x_0,y)>d(x_0,x)
     \end{cases}\\
     &\le d(x,y)\\
     &\le R
  \end{align*}
Now we show $d(x_0,\cdot)$ is coarsely proper. If $B\s \R_{\ge 0}$ is a bounded set then there is some $R\ge 0$ with $B\s B(0,R)$. Then
\begin{align*}
 (d(x_0,\cdot))^{-1}(B)
 &\s \iip{d(x_0,\cdot)}{B(0,R)}\\
 &=B(x_0,R)
\end{align*}
is bounded.
\end{proof}

\begin{defn}\name{asymptotic product}
\label{defn:coarseproduct}
Let $X,Y$ be metric spaces. Fix points $x_0\in X,y_0\in Y$ and a constant $R\ge 0$. The \emph{asymptotic product}\footnote{We guess this notion first appeared in~\cite[chapter~3]{Dranishnikov2000}} with respect to $x_0,y_0,R$ is defined to be
\[
(X,x_0)\ast_R (Y,y_0)=\{(x,y)\in X\times Y:|d(x_0,x)-d(y_0,y)|\le R\}
\]
with the subspace coarse structure from $X\times Y$.
\end{defn}

\begin{prop}
\label{prop:asymptoticproductwelldefined}
Let $X,Y$ be metric spaces, let $x_0,x_0'\in X,y_0,y_0'\in Y$ be points and let $R,R'\ge0$ be numbers. If
\begin{enumerate}
      \item both $X,Y$ are geodesic or
      \item $Y$ is visual from $y_0,y_0'$
\end{enumerate}
then $(X,x_0)\ast_R(Y,y_0)$ and $(X,x_0')\ast_{R'}(Y,y_0')$ are coarsely equivalent. 
\end{prop}
In case $X,Y$ have such nice properties we write short $X\ast Y$ for $(X,x_0)\ast_R(Y,y_0)$. This definition is well defined up to coarse equivalence.
\begin{proof}
Define 
\[
R''=d(x_0,x_0')+d(y_0,y_0')+R
\]
If $(x,y)\in (X,x_0')\ast_R(Y,y_0')$ then 
\begin{align*}
|d(x_0,x)-d(y_0,y)|
&\le |d(x_0,x)-d(x_0',x)|+|d(x_0',x)-d(y_0',y)|+|d(y_0',y)-d(y_0,y)|\\
&\le d(x_0,x_0')+R+d(y,y_0')\\
&=R''
\end{align*}
This implies $(X,x_0')\ast_R(Y,y_0')\s (X,x_0)\ast_{R''}(Y,y_0)$. In the same way we can show there is some $R'''$ such that $(X,x_0)\ast_{R''}(Y,y_0)\s (X,x_0')\ast_{R'''}(Y,y_0)$/ Thus we have shown that $X\ast Y$ is independent of the choice of points if $X\ast Y$ is independent of the choice of constant. We are going to show now the second 
assertion. 

Now we can assume $R'$ is larger than $R$. We show that the inclusion 
\[
 i:(X,x_0)\ast_R(Y,y_0)\to (X,x_0)\ast_{R'}(Y,y_0)
\]
is a coarse equivalence. It is a coarsely injective coarse map obviously.

Suppose both $X,Y$ are geodesic. We show $i$ is coarsely surjective. If 
$(x,y)\in (X,x_0)\ast_{R'}(Y,y_0)$ is not in the 
image of $i$ then
\[
      R<|d(x_0,x)-d(y_0,y)|\le R'.
\]
Assume without loss of generality that $d(x_0,x)>d(y_0,y)$. Choose $x'\in 
\overline{x_0x}$ with $d(x',x_0)=d(y,y_0)$. Here $\overline{x_0x}$ denotes the geodesic segment joining $x_0$ to $x$. Then
\begin{align*}
d(x,x')
&= d(x,x_0)-d(x',x_0)\\
&=d(x,x_0)-d(y,y_0)\\
&\le R'.
\end{align*}
Since $(x',y)\in (X,x_0)\ast_0(Y,y_0)$ and $d((x,y),(x',y))\le R'$ the map $i$ is $R'$-coarsely surjective.

In case $Y$ is visual we can proceed similarly. If $(x,y)\in (X,x_0)\ast_{R'}(Y,y_0)$ is not in the image of $i$ there is some point $y'$ on the geodesic ray extending $\overline{y_0y}$ such that $d(x_0,x)=d(y_0,y')$. Then $d(y,y')\le R'$ as before. Thus $i$ is $R'$-coarsely surjective.
\end{proof}

We denote by $p_X:X\ast Y\to X$ the projection to the first factor and by $p_Y:X\ast Y\to Y$ the projection to the second factor. Note that both $p_X,p_Y$ are coarse maps.
\begin{thm}
\label{thm:pullbackexists}
 Let $X,Y$ be metric spaces with both $X,Y$ geodesic or $Y$ 
visual. Then 
 \[
 \xymatrix{
 X\ast Y \ar[d]_{p_X}\ar[r]^{p_Y}
 & Y\ar[d]^{d(\cdot,y_0)}\\
 X\ar[r]_{d(\cdot,x_0)}
 &\R_+
 }
 \]
 is a limit diagram in the coarse category. Note that we only need the diagram to commute up to closeness.
 \end{thm}
\begin{proof}
 Let $f:Z\to X$ and $g:Z\to Y$ be two coarse maps from a metric space $Z$ such that there is some $R\ge 0$ such that
 \[
 |d(f(z),x_0)-d(g(z),y_0)|\le R.
 \]
Suppose both $X,Y$ are geodesic. For every $z\in Z$ if $d(f(z),x_0)\ge d(g(z),y_0)$ 
there is $\bar f(z)\in Y$ with
 \begin{enumerate}
 \item $d(\bar f(z),x_0)=d(g(z),y_0)$
 \item $d(\bar f(z),f(z))\le R$
 \end{enumerate}
 or if $d(f(z),x_0)< d(g(z),y_0)$ there is $\bar g(z)\in Y$ with
 \begin{enumerate}
 \item $d(f(z),x_0)=d(\bar g(z),y_0)$
 \item $d(\bar g(z),g(z))\le R$.
 \end{enumerate}
 Define 
 \begin{align*}
 \bar f:Z&\to X\\
 z&\mapsto \begin{cases}
                 f(z) & d(f(z),x_0)< d(g(z),y_0)\\
                 \bar f(z) & d(f(z),x_0)\ge d(g(z),y_0).
           \end{cases}
 \end{align*}
 Then $\bar f $ is close to $f$. Define $\bar g:Z\to Y$ in a similar way. Then $\bar g$ is close to $g$. If $X$ is not geodesic but $Y$ is visual we can define $\bar g:Z\to Y$ in a similar way and choose $\bar f =f$.
 Then define 
 \begin{align*}
  \langle f,g\rangle:Z&\to X\ast Y\\
  z&\mapsto (\bar f(z),\bar g(z)).
 \end{align*}
 This map is a coarse map: We show $\langle f,g\rangle$ is coarsely uniform: If $E\s Z^2$ is an entourage then $\zzp f E\s X^2,\zzp g E\s Y^2$ are entourages. Since $\bar g$ is close to $g$ and $\bar f$ is close to $f$ the sets $\zzp {\bar g} E,\zzp {\bar g} E$ are entourages. Then $\zzp {\langle f,g\rangle} E \s \zzp {\bar f} E\times \zzp {\bar g} E$ is an entourage. We show $\langle f,g\rangle$ is coarsely proper: If $B\s X\ast Y$ is bounded then $p_X(B),p_Y(B)$ are bounded. Then
 \[
 \langle f,g\rangle^{-1}(B)\s \ii {\bar f}\circ p_X(B)\cup \ii {\bar g}\circ p_Y(B)
 \]
 is bounded since $f,g $ and thus $\bar f,\bar g$ are coarsely proper. Also $p_X\circ \langle f,g\rangle\sim f$ and $p_Y\circ \langle f,g\rangle\sim 
g$.

Suppose there is another coarse map $h:Z\to (X,x_0)\ast_0 (Y,y_0)$ with the property that 
$p_X\circ h \sim f$ and $p_Y\circ h \sim g$. Then 
\f{
\langle f,g\rangle &\sim\langle p_X\circ h,p_Y\circ h\rangle\\
&=h
}
are close.
\end{proof}

\begin{rem}
 For every metric space $X$ we have $X=(X,x_0)\ast_0(\R_{\ge0},0)$. Thus $\R_{\ge0}$ plays the role of a point.
\end{rem}

\begin{lem}
\label{lem:approper}
 If $X,Y$ are proper geodesic metric spaces then $X\ast Y$ is a proper metric space.
\end{lem}
\begin{proof}
 We show that $X\times Y$ is a proper metric space. If $B\s X\times Y$ is bounded then the projections $B_X$ of $B$ to $X$ and $B_Y$ of $B$ to $Y$ are bounded. Since $X,Y$ are proper the sets $B_X,B_Y$ are relatively compact. Then
\[
 B\s B_X\times B_Y
\]
is relatively compact. Thus $X\times Y$ is proper. Since $X\ast Y\s X\times Y$ is a closed subspace the result follows.
\end{proof}

Since there is no metric space which is final in the coarse category we can repair this by defining a subcategory of the slice category over $\R_{\ge0}$. We chose exactly one morphism for each object $X$, namely $d(\cdot,x_0):X\to \R_{\ge0}$. For $X=\R_{\ge0}$ and $x_0=0$ this is the identity. Thus the object $\R_{\ge0}$ will play the role of the final object. 
\begin{defn}
 A coarse map $\alpha:X\to Y$ between metric spaces is called \emph{weighted} if $d(x_0,\cdot)$ and $d(y_0,\cdot)\circ \alpha$ are close for some $x_0\in X,y_0\in Y$ and hence for all choices of basepoint. The \emph{weighted coarse category} consists of metric spaces as objects and weighted coarse maps modulo close as morphism. 
\end{defn}
This definition makes sense since the composition of two weighted coarse maps is weighted coarse. Not every coarse equivalence is weighted though. Thus the weighted coarse category may not have the same isomorphism classes as the coarse category.

\begin{lem}
\label{lem:weightedproduct}
 In the weighted coarse category the asymptotic product, if well defined, is a product in the category theoretic sense. The space $\R_{\ge 0}$ is a final object in the weighted coarse category.
\end{lem}
\begin{proof}
 We first prove the second statement. If $X$ is a metric space then $d(x_0,\cdot):X\to \R_{\ge 0}$ is weighted coarse. If $\alpha:X\to \R_{\ge 0}$ is another weighted coarse map then $\alpha=\alpha\circ d(0,\cdot)$ is close to $d(x_0,\cdot)$.
 
 Now we prove the first statement. Suppose $X\ast Y$ is the pullback of $d(x_0,\cdot),d(y_0,\cdot)$ in the coarse category. Then in particular the projections $p_X,p_Y$ are weighted coarse. Namely if $(x,y)\in (X,x_0)\ast_R(Y,y_0)$ then
 \begin{align*}
 |d(\cdot,(x_0,y_0))(x,y)-d(\cdot,x_0)\circ p_X(x,y)|
 &=|\max(d(x,x_0),d(y,y_0))-d(x,x_0)|\\
 &\le R.
 \end{align*}
 Thus $p_X$ is weighted. The proof for $p_Y$ is similar. If $Z$ is another metric space with weighted coarse maps $\alpha:Z\to X,\beta:Z\to Y$ then
 \[
  d(x_0,\cdot)\circ \alpha\sim d(z_0,\cdot)\sim d(y_0,\cdot)\circ \beta
 \]
Thus the conditions of Theorem~\ref{thm:pullbackexists} are satisfied. This implies that there exists a unique coarse map $h:Z\to X\ast Y$ such that $p_X\circ h$ is close to $\alpha$ and $p_Y\circ h$ is close to $\beta$. The map $h$ is weighted since
\begin{align*}
 d((x_0,y_0),\cdot)\circ h\sim d(x_0,\cdot)\circ p_X\circ h\sim d(x_0,\cdot)\circ \alpha \sim d(z_0,\cdot).
\end{align*}
\end{proof}

\begin{cor}
 If there are weighted coarse equivalences $\alpha:X\to X',\beta:Y\to Y'$ between geodesic metric spaces then the asymptotic product $X\ast Y$ is coarsely equivalent to $X'\ast Y'$.

 If $\alpha:X\to X'$ is a weighted coarse equivalence between metric spaces and $\beta:Y\to Y'$ is a weighted coarse equivalence between visual metric spaces then there exists a coarse equivalence $X\ast Y\to X'\ast Y'$. 
\end{cor}
\begin{proof}
 If a map is weighted coarse and a coarse equivalence at the same time then it is an isomorphism in the weighted coarse category: Namely if $\alpha'$ is the coarse inverse to $\alpha$ then $\alpha'$ is weighted since
 \[
  |d(\alpha'(x'),x_0)-d(x',x_0')|\le |d(\alpha'(x'),x_0)-d(\alpha\circ \alpha'(x'),x_0')|+|d(\alpha\circ\alpha'(x'),x_0')-d(x',x_0')|
 \]
is bounded. Conversely if a weighted coarse map is an isomorphism then it is in particular a coarse equivalence.

By Lemma~\ref{lem:weightedproduct} the asymptotic product is a product in the weighted coarse category. Since products are well defined up to isomorphism   the result follows. 
\end{proof}
There are properties on metric spaces which are invariant under coarse equivalence. In particular Property A \cite{Willet2009} and Property C \cite{vanMill1984} have been widely studied.
\begin{lem}
 Let $X,Y$ be metric spaces.
 \begin{itemize}
  \item If $X,Y$ have finite asymptotic dimension so does $X\ast Y$.
  \item If $X,Y$ have Property A so does $X\ast Y$.
  \item If $X,Y$ have Property C so does $X\ast Y$.
 \end{itemize}
\end{lem}
\begin{proof}
 Suppose $X,Y$ have finite asymptotic dimension. Then $\asdim(X\times Y)<\infty$ by \cite[Proposition~3.33]{Grave2006}. By \cite[Corollary~3.23]{Grave2006} the subspace $X\ast Y$ has finite asymptotic dimension.
 
 Suppose $X,Y$ have Property A. By \cite[Theorem~I.46]{Lawson2019} the product $X\times Y$ has Property A. By \cite[Proposition~I.34]{Lawson2019} the space $X\ast Y$ has Property A.
 
 Suppose $X,Y$ have Property C. By \cite[Theorem~I.47]{Lawson2019} the space $X\times Y$ has Property C. By \cite[Proposition~I.29]{Lawson2019} the asymptotic product $X\ast Y$ has property C.
\end{proof}

%% file: Homotopy2.tex
\section{Coarse homotopy}

We equip $I:=\R_{\ge 0}\times \R_{\ge 0}$ with the Manhattan metric. Namely 
\[
 d((s_0,t_0),(s_1,t_1))=|s_0-s_1|+|t_0-t_1|
\]
for every $(s_0,t_0),(s_1,t_1)\in I$.

\begin{lem}
 The space $I$ is visual from $(0,0)$.
\end{lem}
\begin{proof}
 If $(x_0,y_0),(x_1,y_1)\in I$ are two points they can be connected by
 \begin{align*}
  \gamma:[0,1]&\to I\\
  t&\mapsto(x_0+t(x_1-x_0),y_0+t(y_1-y_0))
 \end{align*}
If $t,s\in[0,1]$ then
\begin{align*}
 d(\gamma(t),\gamma(s))
 &=|(x_1-x_0)(t-s)|+|(y_1-y_0)(t-s)|\\
 &=(|x_1-x_0|+|y_1-y_0|)|t-s|.
\end{align*}
Thus $\gamma$ is a geodesic. This proves that $I$ is geodesic.

If $\gamma:[a,b]\to I$ is any geodesic we prove: $\gamma(t)_1$ is an increasing or decreasing function and $\gamma(t)_2$ is an increasing or decreasing function. Assume the opposite, there are $a\le t_1<t_2<t_3\le b$ with $\gamma(t_1)_1=\gamma(t_3)_1<\gamma(t_2)_1$. We have
\[
 v|t_3-t_1|=d(\gamma(t_1),\gamma(t_3))=|\gamma(t_1)_2-\gamma(t_3)_2|
\]
and
\begin{align*}
 v|t_2-t_{1/3}|=d(\gamma(t_2),\gamma(t_{1/3}))=|\gamma(t_2)_1-\gamma(t_{1/3})_1|+|\gamma(t_2)_2-\gamma(t_{1/3})_2|>|\gamma(t_2)_2-\gamma(t_{1/3})_2|
\end{align*}
Then
\begin{align*}
 v|t_3-t_1|
 &=|\gamma(t_1)_2-\gamma(t_3)_2|\\
 &\le |\gamma(t_1)_2-\gamma(t_2)_2|+|\gamma(t_2)_2-\gamma(t_3)_2|\\
 &<v|t_2-t_{1}|+v|t_2-t_{3}|\\
 &=v|t_1-t_3|
\end{align*}
is a contradiction.

Now suppose $\gamma:[a,b]\to I$ is a geodesic joining $(0,0)$ to $(x,y)$ with speed $v$. Define 
\begin{align*}
 \tilde \gamma:[a,\infty)&\to I\\
 t&\mapsto\begin{cases}
           \gamma(t) & t\in[a,b]\\
           (x+v(t-b),y) & t\in[b,\infty). 
          \end{cases}
\end{align*}
Clearly $\tilde \gamma$ is a composition of two geodesics. Since $0$ is minimal in both factors of $I$ the function $\gamma$ is increasing in both factors. The same holds for the second geodesic. Since they increase with the same speed they glue at $(x,y)$ to a geodesic ray through $(x,y)$ starting at $(0,0)$.
\end{proof}

\begin{defn}
 Let $X,Y$ be metric spaces.
 \begin{itemize}
  \item A \emph{coarse homotopy} is a coarse map $h:X\ast I\to Y$.
  \item Define $\iota_0:X\to X\ast I$ to be the map $x\mapsto (x,(d(x,x_0),0)$ and $\iota_1:X\to X\ast I$ to be the map $x\mapsto (x,(0,d(x,x_0))$. Two coarse maps $\alpha,\beta:X\to Y$ are \emph{coarsely homotopic} if there is a coarse homotopy $h:X\ast I\to Y$ with $h\circ \iota_0$ close to $\alpha$ and $h\circ \iota_1$ close to $\beta$.
  \item A coarse map $\alpha:X\to Y$ is called a \emph{coarse homotopy equivalence} if there is a coarse map $\beta:Y\to X$ such that $\alpha\circ\beta$ is coarsely homotopic to $id_Y$ and $\beta\circ \alpha$ is coarsely homotopic to $id_X$.
  \item Two coarse spaces $X,Y$ are called \emph{coarsely homotopy equivalent} if there is a coarse homotopy equivalence $\alpha:X\to Y$.
 \end{itemize}
\end{defn}

 \begin{lem}
  This definition of coarse homotopy on coarse maps does not depend on the choice of $R\ge 0$ in $(X,x_0)\ast_R (I,(0,0))$ neither does it depend on the choice of basepoint $x_0\in X$.
 \end{lem}
 \begin{proof}
  If $R\ge 0$ then there exists a coarse equivalence $\psi:(X,x_0)\ast_R (I,(0,0))\to (X,x_0)\ast_0 (I,(0,0))$ which leaves the subspace $(X,x_0)\ast_0 (I,(0,0))\s (X,x_0)\ast_R (I,(0,0))$ invariant. Thus a coarse homotopy $h: (X,x_0)\ast_0 (I,(0,0))\to Y$ joining $h\circ\iota_0$ to $h\circ\iota_1$ gives rise to a coarse homotopy $h\circ \psi: (X,x_0)\ast_R (I,(0,0))\to Y$ joining $h\circ \psi\circ\iota_0=h\circ \iota_0$ to $h\circ \psi\circ \iota_1=h\circ \iota_1$.
  
  Composition with the inclusion $\iota:(X,x_0)\ast_0 (I,(0,0))\to (X,x_0)\ast_R (I,(0,0))$ transforms a coarse homotopy $h:(X,x_0)\ast_R (I,(0,0))\to Y$ joining $h\circ \iota_0$ to $h\circ \iota_1$ to a coarse homotopy $h\circ \iota:X\ast_0 I\to Y$ joining $h\circ\iota\circ\iota_0=h\circ \iota_0$ to $h\circ\iota\circ\iota_1=h\circ \iota_1$.
  
  If $x_1$ is another basepoint then there exists $R\ge 0$ such that $(X,x_0)\ast_R (I,(0,0))$ contains both $(X,x_0)\ast_0 (I,(0,0))$ and $(X,x_1)\ast_0 (I,(0,0))$. Then the maps $\iota_0:x\mapsto (x,(d(x,x_0),0))$ and $\iota_0':x\mapsto (x,(d(x,x_1),0))$ are close to each other and the maps $\iota_1:x\mapsto (x,(0,d(x,x_0)))$ and $\iota_1':x\mapsto (x,(0,d(x,x_1)))$ are close to each other. If $h:(X,x_1)\ast_0 (I,(0,0))\to Y$ is a coarse homotopy joining $h\circ \iota_0'$ to $h\circ \iota_1'$ then $h\circ \psi:(X,x_0)\ast_R (I,(0,0))\to Y$ is a coarse homotopy joining $h\circ\psi\circ \iota_0\sim h\circ \iota'_0$ to $h\circ\psi\circ\iota_1\sim h\circ\iota'_1$.
 \end{proof}

 To prove that two coarse homotopies can be composed we need an auxilary lemma:
 
\begin{lem}
\label{lem:coarselyexcisiveglue}
 Let $X,Y$ be metric spaces. If $A,B\s X$ are a coarsely excisive pair and $\alpha_1:A\to Y,\alpha_2:B\to Y$ are coarse maps with $\alpha_1|_{A\cap B}=\alpha_2|_{A\cap B}$ then they glue to a coarse map $\beta: A\cup B\to Y$. 
\end{lem}
\begin{proof}
 We define a map
 \begin{align*}
  \beta:A\cup B&\to Y\\
  x&\mapsto\begin{cases}
            \alpha_1(x) & x\in A\\
            \alpha_2(x) & x\in B.
           \end{cases}
 \end{align*}
 It remains to show that $\beta$ is coarse. Let $R\ge 0$ be a number. Then there exists $S\ge 0$ with
 \[
  \Delta_R[A]\cap \Delta_R[B]\s \Delta_S[A\cap B].
 \]
If $(x,y)\in \Delta_R\cap (A\times B)$ then $x\in A\cap \Delta_R[B]$. Then $x\in A\cap \Delta_S(A\cap B)$. Thus there exists $u\in A\cap B$ with $(x,u)\in \Delta_S\cap(A\times(A\cap B))$. The triangle inequality tells us that $(u,y)\in \Delta_{R+S}\cap((A\cap B)\times B)$. Thus $(x,y)\in (\Delta_S\cap(A\times(A\cap B)))\circ(\Delta_{R+S}\cap((A\cap B)\times B))$. In a similar way we show 
\[
 \Delta_R\cap (B\times A)\s (\Delta_S\cap (B\times (A\cap B)))\circ(\Delta_{R+S}\cap((A\cap B)\times A))
\]
Now we can show
\begin{align*}
 \zzp\beta{\Delta_R}
 &=\zzp\beta{\Delta_R\cap (A\times A)}\cup \zzp\beta{\Delta_R\cap (A\times B)}\cup \zzp\beta{\Delta_R\cap (B\times A)}\\
 &\quad\cup\zzp\beta{\Delta_R\cap (B\times B)}\\
 &\s\zzp\beta{\Delta_R\cap (A\times A)}\cup (\zzp\beta{\Delta_S\cap(A\times(A\cap B))}\circ\zzp \beta{\Delta_{R+S}\cap((A\cap B)\times B)})\\
 &\quad\cup (\zzp\beta{\Delta_S\cap (B\times (A\cap B))}\circ\zzp\beta{\Delta_{R+S}\cap((A\cap B)\times A)})\cup \zzp\beta{\Delta_R\cap (B\times B)}\\
 &=\zzp{\alpha_1}{\Delta_R\cap (A\times A)}\cup (\zzp{\alpha_1}{\Delta_S\cap(A\times(A\cap B))}\circ\zzp {\alpha_2}{\Delta_{R+S}\cap((A\cap B)\times B)})\\
 &\quad\cup (\zzp{\alpha_2}{\Delta_S\cap (B\times (A\cap B))}\circ\zzp{\alpha_1}{\Delta_{R+S}\cap((A\cap B)\times A)})\cup \zzp{\alpha_2}{\Delta_R\cap (B\times B)}\\
\end{align*}
is an entourage. Thus we have shown that $\beta$ is coarsely uniform. If $B\s Y$ is a bounded set then $\iip \beta B=\iip {\alpha_1} B\cup \iip{\alpha_2}B$ is bounded. Thus $\beta$ is coarsely proper. This way we have shown that $\beta$ is a coarse map.
\end{proof}

\begin{lem}
 Let $X,Y$ be metric spaces. Coarsely homotopic on maps $X\to Y$ is an equivalence relation.
\end{lem}
\begin{proof}
That coarsely homotopic is reflexive and symmetric is obvious. We prove transitive. Let $h_1:X\ast I\to Y$ be a coarse homotopy between $\alpha$ and $\beta$ and let $h_2:X\ast I\to Y$ be a coarse homotopy between $\beta$ and $\gamma$. We define
\begin{align*}
 I_1:=\{(s,t)\in I:s\ge t\}
\quad\mbox{and}\quad
 I_2:=\{(s,t)\in I:t\ge s\}.
\end{align*}
Moreover we define maps
\begin{align*}
\begin{aligned}
 \varphi_1:I_1&\to I\\
 (s,t)&\mapsto\frac{s+t}{s}(s-t,t)
\end{aligned}
& &
\begin{aligned}
 \varphi_2:I_2&\to I\\
 (s,t)&\mapsto\frac{s+t}{t}(s,t-s)
 \end{aligned}
\end{align*}
that are easily seen to be coarse. We define a map 
\begin{align*}
 h:X\ast I&\to Y\\
 (x,(s,t))&\mapsto\begin{cases}
                   h_1(x,\varphi_1(s,t)) & s\ge t\\
                   h_2(x,\varphi_2(s,t)) & t\ge s
                  \end{cases}.
\end{align*}
The composition with $\varphi_1,\varphi_2$ makes sense since 
\begin{align*}
 d(\varphi_1(s,t),(0,0))
 &=\frac{s+t}{s}(|s-t|+|t|)\\
 &=s+t\\
 &=d((s,t),(0,0)).
\end{align*}
A similar calculation shows $\varphi_2$ does not change the distance to $(0,0)$. The pair $X\ast I_1,X\ast I_2$ is coarsely excisive since
\begin{align*}
 \Delta_R[X\ast I_1]\cap \Delta_R[X\ast I_2]
 &=\{(x,(s,t))\in X\ast I:s+R\ge t,t+R\ge s,s+t=d(x,x_0)\}\\
 &=\{(x,(s,t))\in X\ast I:|\frac {d(x,x_0)}{2}-t|\le R/2,|\frac {d(x,x_0)}{2}-s|\le R/2,\\
 &\quad s+t=d(x,x_0)\}\\
 &\s \Delta_{R/2}[\{(x,(s,t))\in X\ast I:s=t=\frac{d(x,x_0)}2\}\\
 &=\Delta_{R/2}[(X\ast I_1)\cap(X\ast I_2)]
\end{align*}
for every $R\ge 0$. If $(x,(d(x,x_0)/2,d(x,x_0)/2))\in(X\ast I_1)\cap(X\ast I_2)$ then
\begin{align*}
 h_1(x,\varphi_1(d(x,x_0)/2,d(x,x_0)/2))
 &=h_1(x,(0,d(x,x_0)))\\
 &=\beta(x)\\
 &=h_2(x,(d(x,x_0),0))\\
 &=h_2(x,\varphi_2(d(x,x_0)/2,d(x,x_0)/2)).
\end{align*}
Thus Lemma~\ref{lem:coarselyexcisiveglue} shows that $h$ is a coarse map. We have
\begin{align*}
 h(x,(d(x,x_0),0))
 &=h_1(x,\varphi_1(d(x_0,x),0))\\
 &=h_1(x,(d(x,x_0),0))\\
 &=\alpha(x)
\end{align*}
and
\begin{align*}
 h(x,(0,d(x,x_0)))
 &=h_2(x,\varphi_2(0,d(x,x_0)))\\
 &=h_2(x,(0,d(x,x_0)))\\
 &=\gamma(x).
\end{align*}
Thus $h$ is a coarse homotopy connecting $\alpha$ and $\gamma$.
\end{proof}

If two coarse maps $\alpha,\beta:X\to Y$ between metric spaces are close then they are obviously coarsely homotopic. Just take the constant homotopy. Note that in practice it is easier to show that two maps are close than to construct a coarse homotopy between them.

%

There are other notions of coarse homotopy which are closely related to the one we just introduced.

We recall the coarse homotopy theory studied in \cite{Mitchener2020}:

If $p:X\to \R_{\ge0}$ is a coarse map then
\[
 I_pX:=\{(x,t)\in X\times \R_{\ge 0}\mid t\le p(x)+1\}
\]
is a cone over $X$. There are inclusions
\begin{align*}
 \begin{aligned}
  i_0:X&\to I_pX\\
  x&\mapsto(x,0)
 \end{aligned}
&\quad&
\begin{aligned}
 i_1:X&\to I_pX\\
 x&\mapsto (x,p(x)+1).
\end{aligned}
\end{align*}
Then two coarse maps $\alpha,\beta:X\to Y$ are said to be coarsely homotopic if there exists a coarse map $h:I_pX\to Y$ with $\alpha=h\circ i_0$ and $\beta=h\circ i_1$.

In some cases, for example if $X$ is a path-metric space, the coarse homotopy  relation on coarse maps $X\to Y$ is independent of the choice of coarse map $p$. Suppose in what follows that $p=d(x_0,\cdot)$.
\begin{prop}
 Two coarse maps $\alpha,\beta:X\to Y$ between metric spaces are coarsely homotopic in our sense exactly when they are coarsely homotopic according to 
\cite{Mitchener2020}.
\end{prop}
\begin{proof}
  We provide maps
 \begin{align*}
 \begin{aligned}
  \Phi:X\ast I&\to I_pX\\
  (x,(s,d(x_0,x)-s))&\mapsto(x,s)
  \end{aligned}&\quad&
  \begin{aligned}
   \Psi:I_pX&\to X\ast I\\
   (x,s)&\mapsto\begin{cases}
                (x,(s,d(x_0,x)-s)) & s\le d(x,x_0)\\
                (x,(d(x_0,x),0)) & \mbox{otherwise}
               \end{cases}
  \end{aligned}
 \end{align*}
 We first show $\Psi$ is the coarse inverse to $\Phi$ by showing these maps are coarsely uniform and their composition is close to the identity. Let $R\ge 0$ be a number. Then
\begin{align*}
 \zzp \Phi{\Delta_R}
 &=\zz\Phi\{((x,(s,d(x,x_0)-s)),(y,(t,d(y,y_0))))\in X\ast I\mid d(x,y)\le R, d(s,t)\le R/2\}\\
 &=\{((x,s),(y,t))\in I_p:d(x,y)\le R,d(s,t)\le R/2\}\\
 &\s \Delta_R
\end{align*}
and
\begin{align*}
 \zzp \Psi{\Delta_R}
 &=\zz \Psi\{((x,s),(y,t))\in I_pX\mid d((x,s),(y,t))\le R\}\\
 &=\zz \Psi\{((x,s),(y,t))\in I_pX\mid d(x,y)\le R,d(s,t)\le R\}\\
 &\s\{((x,(s,d(x,x_0)-s)),(y,(t,d(y,y_0)-t)))\in X\ast I\mid d(x,y)\le R,d(s,t)\le R\}\\
 &\s \Delta_{2R}.
\end{align*}
Thus, $\Phi,\Psi$ are coarsely uniform. We have $\Psi\circ\Phi=id_{X\ast I}$ and
\[
 \Phi\circ\Psi(x,s)=\begin{cases}
                     (x,s)& s\le d(x,x_0)\\
                     (x,d(x,x_0)) & \mbox{otherwise}
                    \end{cases}
\]
is close to the identity on $I_pX$. Thus $\Phi,\Psi$ are coarse inverses. If 
\begin{align*}
 \begin{aligned}
  \iota_0:X&\to X\ast I\\
  x&\mapsto (x,(0,d(x_0,x)))
 \end{aligned}
 &\quad&
\begin{aligned}
 \iota_1:X&\to X\ast I\\
 x&\mapsto (x,(d(x_0,x),0))
\end{aligned}
\end{align*}
are the inclusions then 
\[
\Phi\circ \iota_0(x)=\Phi(x,(0,d(x,x_0)))=(x,0)=i_0(x)
\]
and
\[
\Phi\circ \iota_1(x)=\Phi(x,(d(x,x_0),0))=(x,d(x,x_0)) 
\]
is close to $i_1$. Likewise
\[
 \Psi\circ i_0(x)=\Psi(x,0)=(x,(0,d(x,x_0)))=\iota_0(x)
\]
and
\[
 \Psi\circ i_1(x)=\Psi(x,d(x_0,x)+1)=(x,(d(x,x_0),0))=\iota_1(x).
\]
If $h:I_p X\to Y$ is a coarse homotopy joining $\alpha$ to $\beta$ then $h\circ\Phi:X\ast I\to Y$ is a coarse homotopy joining
\[
h\circ \Phi\circ\iota_0=h\circ i_0=\alpha
\]
to $h\circ \Phi\circ \iota_1$ which is close to $ h\circ i_1=\beta$. By definition the map $\alpha$ is coarsely homotopic to $\beta$ via $h\circ\Phi\circ \iota_1$.

If $g:X\ast I\to Y$ is a coarse homotopy joining $\gamma $ to $\delta$ then
$h\circ\Psi:I_pX\to Y$ is a coarse homotopy joining
\[
 h\circ\Psi\circ i_0=h\circ\iota_0=\gamma
\]
to
\[
 h\circ\Psi\circ i_1=h\circ\iota_1=\delta.
\]
\end{proof}

%% file: Acyclic_Spaces.tex
\section{Acyclic Spaces}
There is a class of metric spaces on which most coarse cohomology theories vanish. Let us translate~\cite[Definition~3.6]{Willet2013} into coarse structure notation:
\begin{defn}
 A metric space $X$ is called \emph{flasque} if there is a coarse map $\phi:X\to X$ such that
 \begin{itemize}
 \item $\phi$ is close to the identity on $X$;
  \item for every bounded set $B\s X$ there is some $N_B\in \N$ such that 
  \[
   \phi^{\circ n}(X)\cap B=\emptyset
  \]
  for every $n> N_B$.
  \item For every entourage $E$ the set $\bigcup_n\zzp{(\phi^{\circ n})}E$ is an 
entourage.
 \end{itemize}
\end{defn}

\begin{thm}
\label{thm:flasqueacyclic}
 If $X$ is a flasque metric space then there is a coarse homotopy equivalence
 \f{
  \varPhi:X\times \Z_{\ge0}&\to X\\
  (x,i)&\mapsto \phi^{\circ i}(x).
 }
 Here $\phi^{\circ 0}$ denotes the identity on $X$.
\end{thm}
\begin{proof}
First we prove that $\varPhi$ is a coarse map. Let $R\ge 0$ be a number. If $d((x,i),(y,j))\le R$ then in particular $d(x,y)\le R$. Then
\[
 \zzp\varPhi{\Delta_R}\s\bigcup_n\zzp{(\phi^{\circ n})}{\Delta_R}
\]
is an entourage. Thus $\varPhi$ is coarsely uniform. If $B\s X$ is a bounded set then there exists some $N\in \N$ with $\im(\phi^{\circ n})\cap B=\emptyset$ for every $n>N$. Then
\[
 \varPhi^{-1}(B)\s \phi^{-\circ0}(B)\cup\cdots\cup \phi^{-\circ N}(B)
\]
is bounded. Thus $\varPhi$ is coarsely proper.

We show that the coarse homotopy inverse to $\varPhi$ is
\f{
i_0:X &\to X\times \Z_{\ge0}\\
x &\mapsto (x,0).
}
We obtain $\varPhi\circ i_0=id_X$. It remains to show that $i_0\circ \varPhi$ and $id_{X\times \Z_{\ge0}}$ are coarsely homotopic. 
Define a map
\begin{align*}
 h:(X\times \Z_{\ge 0})\ast I&\to X\times \Z_{\ge 0}\\
((x,i),(s,\max(d(x,x_0),i)-s))&\mapsto (\phi^{\circ\lfloor \hat s i\rfloor}(x),\lfloor(1-\hat s)i\rfloor).
 \end{align*}
 Here $\hat s:=\frac s{\max(d(x,x_0),i)}$ denotes the normalized parameter. We observe that $h|_{(X\times\Z_{\ge 0})\ast(0\times \Z_{\ge0})}=id_{X\times\Z_{\ge 0}}$ and $h|_{(X\times\Z_{\ge 0})\ast(\Z_{\ge0}\times 0)}=i_0\circ\varPhi$. Now we just need to show that $h$ is a coarse map. First we show the map $h$ is coarsely uniform: Let $R\ge 0$ be a number. If $p_1:X\times \Z_{\ge0}\to X$ denotes the projection to the first factor and $p_2:X\times\Z_{\ge 0}\to \Z_{\ge 0}$ denotes the projection to the second factor then $\zzp h{\Delta_R}$ is an entourage if $\zz {p_1}\circ\zzp h{\Delta_R}$ and $\zz{p_2}\circ\zzp h{\Delta_R}$ are entourages. If $d(((x,i),(s,\max(d(x,x_0),i)-s)),((y,j),(t,\max(d(y,x_0),j)-t)))\le R$ then in particular $d(x,y)\le R,|i-j|\le R$ and $|s-t|\le R$. Now
 \[
  \zz {p_1}\circ \zzp h{\Delta_R}\s \bigcup_n\zzp{(\phi^{\circ n})}{\Delta_R}
 \]
is an entourage. 
\begin{align*}
d((p_2\circ h((x,i),(s,s')),p_2\circ h((y,j),(t,t')))
&=|\lfloor (1-\hat s)i\rfloor-\lfloor (1-\hat t)j\rfloor|\\
&\le |(1-\hat s)i-(1-\hat t)j|+2\\
&\le |i-j|+|\hat si-\hat tj|+2\\
&\le |i-j|+|\hat s i -\hat t i|+|\hat ti-\hat t j|+2\\
&\le R+R+R+2
\end{align*}
Thus $\zz{p_{2}}\circ \zz h(\Delta_R)$ is an entourage.

Now we show that $h$ is coarsely proper: Let $B\s X\times \Z_{\ge0}$ be a bounded subset. We 
write
 \[
  B=\bigcup_i B_i\times i
 \]
 which is a finite union. Then for every $i$ there is some $N_i$ such that
 \[
 \phi^{\circ n}(X)\cap B_i=\emptyset
 \]
for every $n> N_i$. The set $\iip h B$ is bounded if $p_1\circ p_{X\times \Z_{\ge 0}}\circ h^{-1}(B)$ and $p_2\circ p_{X\times \Z_{\ge 0}}\circ h^{-1}(B)$ are bounded. Now
\[
p_1\circ p_{X\times\Z_{\ge 0}}\circ \iip h B\s \bigcup_i (\phi^{\circ-0}(B_i)\cup\cdots\cup\phi^{\circ- 
N_i}(B_i))
\]
is bounded in $X$. If $j\in p_2\circ p_{X\times \Z_\ge0}\circ \iip h B$ then $\lfloor tj\rfloor \le 
N_i$ for at least one $i$ and some $t\in[0,1]$. Thus
\[
 j\le \max_i N_i
\]
 is bounded in $\Z_{\ge0}$.
\end{proof}

\begin{ex}
 Note that $\Z_{\ge 0}$ is flasque by
 \f{
 \phi:\Z_{\ge0}&\to \Z_{\ge0}\\
  n&\mapsto n+1.
 }
 Thus there is a coarse homotopy equivalence ${\Z_{\ge0}}^2\to \Z_{\ge0}$. Now for every $n$ the space ${\Z_{\ge0}}^n$ is flasque. As a result ${\Z_{\ge 0}}^n$ is coarsely homotopy equivalent to $\Z_{\ge0}$ for every $n$.
\end{ex}

%% file: Colimits.tex
\section{Finite colimits}

\begin{defn}
 Let $(A,d_A),(B,d_B)$ be two nonempty metric spaces. The \emph{coarse disjoint union $A\vee B$ of $A,B$} is defined to be
 \[
  A\sqcup B/a_0\sim b_0
 \]
where $a_0\in A,b_0\in B$ are two fixed points and
\begin{align*}
 d:(A\vee B)\times(A\vee B)&\to \R\\
 (x,y)&\mapsto\begin{cases}
               d_A(x,y) & x,y\in A\\
               d_A(x,a_0)+d_B(b_0,y) & x\in A,y\in B\\
               d_B(x,b_0)+d_A(a_0,y) & x\in B,y\in A\\
               d_B(x,y) & x,y\in B
              \end{cases}.
\end{align*}

\end{defn}

In \cite{Weighill2016} a similar Definition has been made. They ar equivalent since they share the same universal property:
\begin{lem}
\label{lem:coproduct}
 If $A,B$ are metric spaces then $A\vee B$ is a metric space. It is a coproduct of $A,B$.
\end{lem}
\begin{proof}
First we prove that $A\vee B$ is a metric space. If $x,y\in A\vee B$ are two points with $d(x,y)=0$ and both $x,y\in A$ or both $x,y\in B$ then $x=y$ is clear. If $x\in A$ and $y\in B$ then $0=d(x,y)=d(x,a_0)+d(b_0,y)$ implies $d(x,a_0)=0$ and $d(b_0,y)=0$ thus $x=a_0=b_0=y$. Symmetry of $d$ is clear. We prove the triangle inequality. If $x,y\in A$ and $z\in B$ then
\begin{align*}
 d(x,y)+d(y,z)
 &=d(x,y)+d(y,a_0)+d(b_0,z)\\
 &\ge d(x,a_0)+d(b_0,z)\\
 &=d(x,z).
\end{align*}
Now we prove $A\vee B$ with the natural inclusions $i:A\to A\vee B$ and $j:B\to A\vee B$ is a coproduct. Suppose $C$ is a metric space and $\alpha:A\to C,\beta:B\to C$ are coarse maps. Then define
\begin{align*}
 \gamma:A\vee B&\to C\\
 x&\mapsto\begin{cases}
           \alpha(x) & x\in A\\
           \beta(x) & x\not\in A.
          \end{cases}
\end{align*}
Then $\gamma\circ i=\alpha$ and $\gamma\circ j$ is close to $\beta$. It remains to show that $\gamma$ is a coarse map. If $R\ge 0$ then there exist $\alpha(R),\beta(R)\ge0$ with $\zzp\alpha{\Delta_R}\s\Delta_{\alpha(R)}$ and $\zzp \beta{\Delta_R}\s\Delta_{\beta(R)}$. If $x,y\in A\vee B$ with $d(x,y)\le R$ then there are 4 cases. If both $x,y\in A$ then $d(\gamma(x),\gamma(y))=d(\alpha(x),\alpha(y))\le\alpha(R)$. If both $x,y\not\in A$ then $d(\gamma(x),\gamma(y))=d(\beta(x),\beta(y))\le \beta(R)$. If $x\in Ay\not\in A$ then
\begin{align*}
 d(\gamma(x),\gamma(y))
 &=d(\alpha(x),\beta(y))\\
 &\le d(\alpha(x),\alpha(a_0))+d(\alpha(a_0),\beta(b_0))+d(\beta(b_0),\beta(y))\\
 &\le \alpha(R)+d(\alpha(a_0),\beta(b_0))+\beta(R).
\end{align*}
The case $x\not\in A,y\in A$ is similar. Thus $\gamma$ is coarse.

If $\delta:A\vee B\to C$ is another coarse map which fits in the pushout diagram then $\delta|_A=\delta\circ i\sim\alpha=\gamma|_A$ and $\delta|_B=\delta\circ j \sim \beta\sim \gamma|_B$. Thus $\delta\sim \gamma$ is unique.

\end{proof}
 Two subsets $U,V\s X$ form a coarse disjoint union of $X$ if for every $R\ge0$ the sets $\Delta_R[U^c]\cap \Delta_R[V^c], U\cap V$ are bounded.
\begin{lem}
 If $U,V$ is a coarse disjoint union of a metric space $X$ then $X, U\vee V$ are coarsely equivalent.
\end{lem}
\begin{proof}
 We define a map 
 \begin{align*}
  \varphi:X&\to U\vee V\\
  x&\mapsto \begin{cases}
             [x] & x\in U\cup V\\
             [u_0] & x\in (U\cup V)^c.
            \end{cases}
 \end{align*}
  It remains to show that $\varphi$ is a coarse equivalence. First we show $\varphi$ is coarsely uniform. If $R\ge 0$ and $d_X(x,y)\le R$ we distinguish three cases. If $(x,y)\in U^2$ then $d_{U\vee V}([x],[y])= d_U(x,y)\le R$. If $x\in U,y\in V$ then since $\Delta_R[U]\cap\Delta_R[V]$ is bounded there exists $S_R\ge 0$ with $d_U(x,u_0)\le S_R,d_V(y,v_0)\le S_R$. Then 
  \[
  d_{U\vee V}([x],[y])=d_U(x,u_0)+d_V(v_0,y)\le 2S_R.
  \]
  If $x\in (U\cup V)^c,y\in U$ then since $(U\cup V)^c$ is bounded there exists some $S\ge 0$ with $d_X(x,u_0)\le S$. Then 
  \[
  d_{U\vee V}(\varphi(x),\varphi(y))=d_X(u_0,y)\le d_X(u_0,x)+d_X(x,y)\le S+R.
  \]
  The other cases are similar. Thus $\varphi$ is coarsely uniform. 
  
 Now we show that $\varphi$ is coarsely injective. If $R\ge 0$ and $d(\varphi(x),\varphi(y))\le R$ we again distinguish three cases. If $x,y\in U$ then $d_X(x,y)=d_{U\vee V}([x],[y])\le R$. If $x\in U,y\in V$ then $d([x],[y])=d(x,u_0)+d(v_0,y)$. Then
 \[
  d(x,y)\le d(x,u_0)+d(u_0,v_0)+d(v_0,y)\le R+d(u_0,v_0).
 \]
If $x\in (U\cup V)^c,y\in U$ then since $(U\cup V)^c$ is bounded, there exists some $S\ge 0$ with $d(x,u_0)\le S$. Then 
\[
 d(x,y)\le d(x,u_0)+d(u_0,y)\le S+R.
\]
The other cases are similar.
 
 Since $\varphi$ is surjective we have proven that is a coarse equivalence.
\end{proof}

In the following definition we use there is a more general concept of coarse structure. It has to satisfy a certain set of axioms \cite{Roe2003}. The coarse structure of a metric space is an example for this. We have that the intersection of coarse structures is again a coarse structure so we can give a collection of subsets of $X\times X$ and talk about the coarse structure they generate.
\begin{defn}
 Let $\alpha,\beta:X\to Y$ be two coarse maps between metric spaces. Then the space $Y(\alpha,\beta)$ is defined to be $Y$ as a set equiped with the coarse structure generated by $(\Delta_R)_R$ and
 \[
  \Delta(\alpha,\beta)=\{(\alpha(x),\beta(x))\mid x\in X\}.
 \]
The \emph{asymptotic coequalizer of $\alpha,\beta$} is $q:=id_Y:Y\to Y(\alpha,\beta)$.
\end{defn}

\begin{prop}
\label{prop:coequalizer}
 If $\alpha,\beta:X\to Y$ are two coarse maps between metric spaces then $Y(\alpha,\beta)$ is a coequalizer of $\alpha,\beta$ in the category of metric spaces and coarse maps modulo close.
\end{prop}
\begin{proof}
We first remark that $Y(\alpha,\beta)$ is a metric space since its coarse structure is countably generated \cite{Roe2003}.

 Now we prove that $q$ is a coarse map. Since every entourage in $Y$ is an entourage in $Y(\alpha,\beta)$ the map $q$ is coarsely uniform. If $z\in Y$ then
 \begin{align*}
  \Delta(\alpha,\beta)[z]&=\{y\in Y\mid\exists x\in X,\alpha(x)=y,\beta(x)=z\}\\
  &=\{\alpha(x)\mid x\in \beta^{-1}(z)\}\\
  &=\alpha\circ\beta^{-1}(z)
 \end{align*}
is bounded. Similarly we obtain
\begin{align*}
 \Delta(\alpha,\beta)^{-1}[z]
 &=\Delta(\beta,\alpha)[z]\\
 &=\beta\circ\alpha^{-1}(z)
\end{align*}
is bounded. By an inductive argument every bounded set in $Y(\alpha,\beta)$ is already bounded in $Y$. Thus $q$ is coarsely proper. Now $(q\circ \alpha\times q\circ \beta)(\Delta_X)=\Delta(\alpha,\beta)$ is an entourage in $Y(\alpha,\beta)$. Thus $q\circ \alpha,q\circ\beta$ are close.

Now we prove the universal property. If $r:Y\to Z$ is a coarse map with $r\circ\alpha$ close to $r\circ \beta$ then $\{(r\circ \alpha(x),r\circ \beta(x)\mid x\in X\}$ is an entourge. Define 
\begin{align*}
r':Y(\alpha,\beta)&\to Z\\
y&\mapsto r(y).
\end{align*}
If $R\ge0$ then $\zzp{r'}{\Delta_R}$ is an entourage since $r=r'$. Moreover
\[
 \zzp {r'}{\Delta(\alpha,\beta)}=\{(r\circ \alpha(x),r\circ \beta(x)\mid x\in X\}
\]
is an entourage. Thus $r'$ is coarsely uniform. Since every bounded set in $Y$ is bounded in $Y(\alpha,\beta)$ the map $r'$ is coarsely proper. Thus $r'$ is coarse with $r= r'\circ q$. If $r'':Y(\alpha,\beta)\to Z$ is another coarse map with $r\sim r''\circ q$ then $r''\sim r'$.
\end{proof}

\begin{thm}
 The category of metric spaces and coarse maps modulo close has all finite colimits.
\end{thm}
\begin{proof}
 By Lemma~\ref{lem:coproduct} the category has all finite coproducts. By Proposition~\ref{prop:coequalizer} the category has all coequalizers. The claim is then a general result in category theory.
\end{proof}